\documentclass[a4paper]{article}
\usepackage{graphicx} %
\usepackage{amsmath}
\usepackage{amssymb}
\usepackage{amsthm}
\usepackage{enumerate}
\usepackage{complexity}
\usepackage{xcolor}
\usepackage[a4paper,left=3.5cm,right=3.5cm,top=3cm,bottom=4cm]{geometry}
\usepackage{tabularx}
\usepackage[T1]{fontenc}
\usepackage{authblk}
\usepackage{subcaption}
\usepackage{hyperref}
\usepackage[capitalize,noabbrev]{cleveref}

\newtheorem{theorem}{Theorem}

\newtheorem{lemma}[theorem]{Lemma}

\newtheorem{corollary}[theorem]{Corollary}
\newtheorem{question}[theorem]{Question}

\newtheorem{conjecture}[theorem]{Conjecture}

\def\Cr{\operatorname{cr}} %

\newcommand{\Prob}[3]{\begin{center}
\renewcommand{\arraystretch}{0.85}%
\begin{tabularx}{\textwidth}{|lX|}
	\hline
	\multicolumn{2}{|l|}{\hspace{0pt}\rule{0pt}{13pt}\large\sc#1}\\
	{\bf Input:\enspace}&{#2}\\
	{\bf Question:\enspace}&{#3\rule[-6pt]{0pt}{6pt}} \\
	\hline
\end{tabularx}
\end{center}
}

\newcommand{\Ucr}{\ensuremath{\mathrm{unc}}}
\newcommand{\UcrProblem}{{\sc Uncrossed\-Number}}

\newcommand{\ecrProblem}{{\sc Edge\-Crossing\-Number}}
\newcommand{\mucsgProblem}{{\sc Maximum\-Uncrossed\-Subgraph}}

\usepackage{tikz}
\usetikzlibrary{shapes.geometric}
\usetikzlibrary{positioning}
\usetikzlibrary{math,calc}
\usetikzlibrary{decorations.pathmorphing,decorations.pathreplacing}
\usetikzlibrary{backgrounds}

\tikzstyle{origV} = [circle,draw,fill=gray!42,inner sep=1pt,minimum size=13pt]
\tikzstyle{gadgetV} = [draw,inner sep=1pt,minimum size=13pt,scale=.8]
\tikzstyle{helperV} = [circle,fill=black,scale=.3]
\tikzstyle{centerV} = [circle,draw,fill=black,text=white,inner sep=0.5pt]
\tikzstyle{minihelperV} = [circle,draw=black,fill=white,scale=.25]
\tikzstyle{origE} = [thick,decorate]
\tikzstyle{gadgetE} = [thick]
\tikzstyle{splitE} = [line width=0.8pt,densely dashed]

\tikzstyle{bendRsmall} = [bend right=12,looseness=1.7]
\tikzstyle{bendRmed} = [bend right=25,looseness=1.7]
\tikzstyle{bendRmedbig} = [bend right=40,looseness=1.7]
\tikzstyle{bendRbig} = [bend right=85,looseness=1.5]

\tikzstyle{bendLsmall} = [bend left=12,looseness=1.7]
\tikzstyle{bendLmed} = [bend left=25,looseness=1.7]
\tikzstyle{bendLmedbig} = [bend left=40,looseness=1.7]
\tikzstyle{bendLbig} = [bend left=85,looseness=1.5]

\newcommand\radius{2.3}
\newcommand\boundingbox[1][1.3*\radius]{\path[use as bounding box] ($(#1,#1)$) rectangle ($(-#1, -#1)$);}

\newcommand\gateedge[2][]{%
	\draw[splitE,#1] #2 to (0,0);
}

\newcommand\origedge[5][black]{
	\draw[#1,#4,ultra thick] (#2) to (#3);
}

\newcommand\origedgeDetailed[5][black]{
	\foreach \pos in {#5}{
		\foreach \i in {-3.0,-1.8,-0.6,3.0}{
			\draw[#1,#4] (#2) -- ($(#2)!\pos!(#3)!0.1*\i cm!90:(#3)$) node [circle,fill=black,scale=0.25] {};
			\draw[black,#4] ($(#2)!\pos!(#3)!0.1*\i cm!90:(#3)$) -- (#3);
		}
		\foreach \i in {0.6,1.2,1.8}{
			\path        (#2) -- ($(#2)!\pos!(#3)!0.1*\i cm!90:(#3)$) node [circle,fill=black,scale=0.1] {} -- (#3);
		}
	}
}

\title{On the Uncrossed Number of Graphs\thanks{A conference short version of this paper will appear in the proceedings of the 32nd International Symposium on Graph Drawing and Network Visualization (GD 2024),
    Schloss Dagstuhl --- Leibniz-Zentrum für Informatik, %
    Germany, 2024, pp.\ 18:1--18:13~\cite{GD24}.
}}

\author[1]{Martin Balko\thanks{Supported by grant no.\ 23-04949X of the Czech Science Foundation (GA\v{C}R) and by the Center for Foundations of Modern Computer Science (Charles Univ. project UNCE 24/SCI/008).}}
\author[2]{Petr Hliněný}
\author[3]{Tomáš Masařík\thanks{Supported by the Polish National Science Centre SONATA-17 grant number 2021/43/D/ST6/03312.}}
\author[4]{Joachim Orthaber\thanks{Supported by the Austrian Science Fund (FWF) grant W1230.}}
\author[4]{Birgit~Vogtenhuber}
\author[5]{Mirko H.~{Wagner}}
\affil[1]{Department of Applied Mathematics, Faculty of Mathematics and Physics, Charles University, Czech Republic}
\affil[2]{Faculty of Informatics, Masaryk University, Brno, Czech Republic}
\affil[3]{Institute of Informatics, Faculty of Mathematics, Informatics and Mechanics, University of Warsaw, Poland}
\affil[4]{Institute of Software Technology, Graz University of Technology, Austria}
\affil[5]{Institute of Computer Science, Osnabrück University, Germany}

\date{}

\begin{document}

\maketitle

\begin{abstract}

Visualizing a graph $G$ in the plane nicely, for example, without crossings, is unfortunately not always possible.
To address this problem, Masařík and Hliněný [GD 2023] recently asked for each edge of $G$ to be drawn without crossings while allowing multiple different drawings of $G$.
More formally, a collection~$\mathcal{D}$ of drawings of $G$ is \emph{uncrossed} if, for each edge $e$ of $G$, there is a drawing in $\mathcal{D}$ in which $e$ is uncrossed.
The \emph{uncrossed number} $\Ucr(G)$ of $G$ is then the minimum number of drawings in some uncrossed collection of $G$.

No exact values of the uncrossed numbers have been determined yet, not even for simple graph classes.
In this paper, we provide the exact values for uncrossed numbers of complete and complete bipartite graphs, partly confirming and partly refuting a conjecture posed by Hliněný~and Masařík~[GD 2023].
We also present a strong general lower bound on $\Ucr(G)$ in terms of the number of vertices and edges of $G$.
Moreover, we prove \NP-hardness of the related problem of determining the \emph{edge crossing number} of a graph $G$, which is the smallest number of edges of $G$ taken over all drawings of $G$ that participate in a crossing.
This problem was posed as open by Schaefer in his book [Crossing Numbers of Graphs 2018].

\smallskip

\begin{center}
\end{center}
  
\end{abstract}

\section{Introduction}

In a \emph{drawing} of a graph $G$, the vertices are represented by distinct points in the plane and each edge corresponds to a simple continuous arc connecting the images of its end-vertices.
As usual, we identify the vertices and their images, as well as the edges and the line segments representing them.
We require that the edges
pass through no vertices other than their endpoints. We assume for simplicity that any
two edges have only finitely many points in common, no two edges touch at an interior point
and no three edges meet at a common interior point.

A \emph{crossing} in a drawing $D$ of $G$ is a common interior point of two edges of $D$ where they properly cross.
For a drawing $D$ of a graph $G$, we say that an edge $e$ of $D$ is \emph{uncrossed} in $D$ if it does not share a crossing with any other edge of $D$.

There are two staple problems in the graph drawing field that defined the past eighty years of development in the area.
The first one, dating back to World War~II times~\cite{WWIIhistory,Zarankiewicz1955}, is the problem of determining the \emph{crossing number} $\Cr(G)$ of a graph $G$, defined as the smallest number of crossings required in any drawing of $G$ in the plane.
The crossing number problem has been intensively studied ever since, especially in the past thirty years.
Computing the crossing number is \NP-hard on general graphs~\cite{GareyJ83}, and one can find a thorough overview of the area in a recent book by Schaefer~\cite{Schaefer18}.

The second, only slightly newer problem, is that of determining the \emph{thickness} $\theta(G)$ of a graph $G$, defined as the smallest integer $k$ such that $G$ can be edge-partitioned into $k$ planar graphs.
This problem was proposed by Harary~\cite{Har61} in 1961 and since then this concept has played an important role in graph drawing. 
Unlike for planarity, deciding whether a graph is \emph{biplanar}, that is whether $\theta(G)\leq 2$, is \NP-complete~\cite{Mansfield}.
For an overview of the progress up to 1998, consult a survey by Mutzel, Odenthal, and Scharbrodt~\cite{ThickSurvey}.

In this paper, we investigate a very recent notion inspired by a fusion of both concepts into one.
We say that a collection $\mathcal{D}(G)$ of drawings of $G$ is \emph{uncrossed} if for each edge $e$ of $G$ there is at least one drawing in $\mathcal{D}(G)$ in which $e$ is uncrossed; see \Cref{fig-uncrossed} for an example.
Hliněný and Masařík~\cite{masHlin23}, in relation to extensions of the traditional crossing number of a graph, defined the \emph{uncrossed number} $\Ucr(G)$ of a graph $G$ as the smallest size of an uncrossed collection of drawings of $G$.
\begin {figure}[t]
  \centering
  \includegraphics{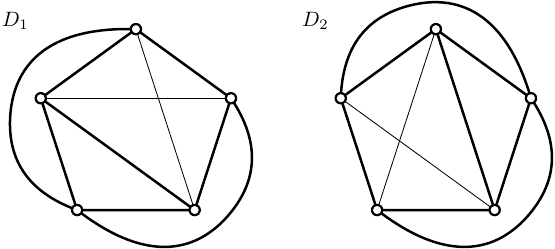}
  \caption {An uncrossed collection $\mathcal{D}(K_5) = \{D_1,D_2\}$ that shows $\Ucr(K_5) \leq 2$. The edges that are uncrossed are shown in thick lines.
  Since $K_5$ is not planar, we have $\Ucr(K_5) = 2$.}
  \label{fig-uncrossed}
\end {figure}
The motivation for the uncrossed number is that finding a handful of different drawings of a graph $G$ instead of just one ``flawless'' drawing shall highlight different aspects of $G$ and thus could be useful for the visualization of $G$, besides the theoretical interest.
The requirement that each edge is uncrossed in at least one drawing is a natural way to enforce that the drawings will highlight each part of the graph while still displaying the whole graph.

Let us also formulate the decision version of the problem of determining $\Ucr(G)$ of a given graph $G$.
\Prob{UncrossedNumber}{A graph $G$ and a positive integer $k$.}{Are there drawings $D_1,\ldots,D_k$ of $G$ such that, for each edge $e\in E(G)$, there exists an $i\in[k]$ such that $e$ is uncrossed in the drawing $D_i$?}

Clearly, for every graph $G$, we have
\begin{equation}
\label{eq-thickness}
\Ucr(G) \geq \theta(G),
\end{equation}
because the uncrossed edges in each drawing of an uncrossed collection of $G$ induce an edge-partition of $G$ into planar graphs.
However, this new concept significantly differs from thickness (which just partitions the edges of $G$) in the sense that all edges of $G$ have to be present along with the uncrossed subdrawing in each drawing of our uncrossed collection.
In fact, the requirements of an uncrossed collection bring us close to the related notion of the \emph{outerthickness} $\theta_o(G)$ of $G$, which is the minimum number of outerplanar graphs into which we can decompose $G$.

\Prob{Outerthickness}{A graph $G$ and a positive integer $k$.}{Can $G$ be decomposed into $k$ outerplanar graphs?}%

As noted by Hlin\v{e}n\'{y} and Masa\v{r}\'{i}k~\cite{masHlin23}, given a decomposition $\{G_1,\dots,G_k\}$ of $G$ into outerplanar graphs, we can let $D_i$ be an outerplanar drawing of $G_i$ with all remaining edges of $G$ being drawn in the outer face.
This gives us 
\begin{equation}
\label{eq-outerthickness}
\Ucr(G) \leq \theta_o(G)
\end{equation}
for every graph $G$.
Combining this with a result of Gon\c{c}alves~\cite{Goncalves2}, which implies $\theta_o(G)\leq2\theta(G)$, we actually obtain the following chain of inequalities
\begin{equation}
  \frac12\theta_o(G) \leq \theta(G) \leq \Ucr(G) \leq \theta_o(G) \leq 2 \theta(G).
\end{equation}

So far, the exact values of uncrossed numbers are not very well understood.
Masařík and Hliněný~\cite{masHlin23} exactly determined $\Ucr(G)$ of only a few sporadic examples of graphs $G$, such as $\Ucr(K_7)=3$.

\subsection{Our Results}

We determine the exact values of uncrossed numbers for specific and natural graph classes.
First, we derive a formula for the uncrossed number of complete graphs.
\begin{theorem}
\label{thm:complete}
For every positive integer $n$, it holds that %
\[
\Ucr(K_n) = 
\begin{cases}
    \lceil \frac{n+1}{4} \rceil, & \text{for } n \notin \{4,7\} \\
    3, & \text{for } n=7\\
    1, & \text{for } n=4.
\end{cases}
\]
\end{theorem}

We also find the exact formula for the uncrossed number of complete bipartite graphs.%

\begin{theorem}
\label{thm:completeBip}
For all positive integers $m$ and $n$ with $m \leq n$, it holds that %
\[
  \Ucr(K_{m,n}) = 
  \begin{cases}
    \lceil\frac{mn}{2m+n-2}\rceil, & \text{for }3 \leq m \leq n \leq 2m-2 \\
    \lceil\frac{mn}{2m+n-1}\rceil, & \text{for } n=2m-1 \\
    \lceil\frac{mn}{2m+n}\rceil, & \text{for } 6 \leq 2m \leq n \\
    1, & \text{for } m \leq 2.%
  \end{cases}
\]
\end{theorem}

Let us mention that the exact values of the thickness $\theta(K_{m,n})$ of complete bipartite graphs are not known for all values of $m$ and $n$;  see~\cite{poranMakin05} for further discussion.

We compare our formulas on $\Ucr(K_n)$ and $\Ucr(K_{m,n})$ with known formulas on the outerthickness of $K_n$ and $K_{m,n}$.
Hliněný and Masařík~\cite[Section 6]{masHlin23} conjectured the uncrossed numbers and outerthickness to be the same for both complete and complete bipartite graphs except in the planar but not outerplanar cases.
Guy and Nowakowski~\cite{guyNow90I,guyNow90II} showed that
\begin{equation}
\label{eq-outerthicknessKn}
\theta_o(K_n) = 
\begin{cases}
    \left\lceil \frac{n+1}{4} \right\rceil, & \text{for } n \neq 7 \\
    3, & \text{for } n=7
\end{cases}
\end{equation}
and 
\begin{equation}
\label{eq-outerthicknessKmn}
\theta_o(K_{m,n}) = \left\lceil \frac{mn}{2m+n-2} \right\rceil
\end{equation}
for all positive integers $m$ and $n$ with $m \leq n$. 
Note that it follows from \Cref{thm:complete,eq-outerthicknessKn} that $\Ucr(K_n)=\theta_o(K_n)$ for every $n\neq 4$.
For $n=4$, we have $\Ucr(K_4)=1$ while $\theta_o(K_4)=2$.
This confirms the conjecture of Hliněný and Masařík~\cite{masHlin23} in the case of complete graphs.

Since 
\[\left\lceil\frac{mn}{2m+n}\right\rceil\le  \left\lceil\frac{mn}{2m+n-2}\right\rceil = \left\lceil\frac{mn}{2m+n} + \frac{2mn}{(2m+n-2)(2m+n)} \right\rceil \leq \left\lceil\frac{mn}{2m+n}\right\rceil + 1\]
for $n \geq 2m-1>1$, it follows from \Cref{thm:completeBip} and \Cref{eq-outerthicknessKmn} that the uncrossed number $\Ucr(K_{m,n})$ differs from the outerthickness $\theta_o(K_{m,n})$ of $K_{m,n}$ by at most 1.
In particular, we have $\Ucr(K_{n,n})=\theta_o(K_{n,n})$ for every positive integer $n$.
However, Theorem~\ref{thm:completeBip} and~\eqref{eq-outerthicknessKmn} give, for example, $\Ucr(K_{4,7})= 2$ and $\theta_o(K_{4,7})=3$.
Since $K_{4,7}$ is not planar, this refutes the conjecture of Hliněný and Masařík~\cite{masHlin23} in the case of complete bipartite graphs.

\smallskip
Second, we turn our attention to general graphs and their uncrossed number.
We improve the trivial lower bound of $\Ucr(G) \geq \lceil m/(3n-6)\rceil$ for any graph $G$ with $n$ vertices and $m$ edges.
By carefully balancing between the numbers of edges in uncrossed subdrawings of~$G$ and the numbers of edges that can be drawn within faces of uncrossed subdrawings, we derive the following estimate.

\begin{theorem}
\label{thm-lowerBound}
Every connected graph $G$ with $n \geq 3$ vertices and $m \geq 0$ edges satisfies
\[
\Ucr(G) \geq \left\lceil\frac{m}{f(n,m)}\right\rceil
\]
where $f(n,m) = (3n-5+\sqrt{(3n-5)^2-4m})/2$.
\end{theorem}

The bound from Theorem~\ref{thm-lowerBound} becomes interesting for $m \geq 3n-6$.
This is because we then have $f(n,m) \leq 3n-6$ for all integers $n \geq 2$ as 
\[\sqrt{(3n-5)^2-4m} \leq \sqrt{9n^2-42n+49} = 3n-7\] for any $m \geq 3n-6 \geq 0$.
It follows that the lower bound from \Cref{thm-lowerBound} is at least as good as $\Ucr(G) \geq \lceil m/(3n-6)\rceil$ for any connected $G$ with $n \geq 2$ vertices and $m \geq 3n-6$ edges.

The lower bound from \Cref{thm-lowerBound} gets stronger as the graph $G$ gets denser.
For example, if $G$ contains $n$ vertices and $\varepsilon n^2$ edges for $n$ sufficiently large and $\varepsilon \in (0,1/2)$, we get
\[f(n,m) = (3n-5+\sqrt{(9-4\varepsilon)n^2 -30n + 25})/2 \leq (3+\sqrt{9-4\varepsilon}) n/2.\]
Since $(3+\sqrt{9-4\varepsilon})/2 < 3$ for $\varepsilon > 0$, we obtain $\Ucr(G) \geq \left\lceil\frac{m}{c_\varepsilon n}\right\rceil$ for any $\varepsilon > 0$ and some constant $c_\varepsilon<3$, instead of trivial $\Ucr(G) \geq \left\lceil\frac{m}{3n-6}\right\rceil$.
We note that the best constant $c_\varepsilon$ obtainable from Theorem~\ref{thm-lowerBound} is $(3+\sqrt{7})/2 \sim 2.82$ as $\varepsilon \leq 1/2$.

\smallskip

We also consider computational complexity aspects related to the {\sc UncrossedNumber} problem. 
As we will see later, a closely related problem is the one of determining the \emph{edge crossing number} of a given graph $G$, which is the smallest number of edges involved in crossings taken over all drawings of $G$.
The notion of the edge crossing number is based on results by Ringel~\cite{ringel64}, Harborth and Mengersen~\cite{haMen74,haMen90}, and Harborth and Thürmann~\cite{haThu96}.

\Prob{\ecrProblem}{A graph $G$ and a positive integer $k$.}{Is there a drawing $D$ of $G$ with at most $k$ edges involved in crossings?}

The complementary problem to \ecrProblem{} is the following one.

\Prob{\mucsgProblem}{A graph $G$ and a positive integer $k$.}{Is there a drawing $D$ of $G$ with at least $k$ edges not involved in any crossings?}

In his monograph on crossing numbers, Schaefer~\cite{Schaefer18} mentions that the problem of determining the computational complexity of \ecrProblem{} is open.
Here, we resolve this open question by showing that the problem is \NP-complete.

\begin{theorem}
	\label{theorem:hard:ecr}
	The \ecrProblem{} problem is \NP-complete.
\end{theorem}

By the complementarity of the problems \mucsgProblem{} and \ecrProblem{}, we obtain the following result.

\begin{corollary}
	The \mucsgProblem{} problem is \NP-complete. 
\end{corollary}

As a consequence of our reduction, we also obtain the following relative result. 
\begin{theorem}
\label{thm-hardness}
	If the {\sc Outerthickness} problem is \NP-hard, then also the \UcrProblem{} problem is \NP-hard.
\end{theorem}
However, in contrast to the complexity of the {\sc thickness} problem, which was shown to be \NP-hard already in 1983 by Mansfield~\cite{Mansfield}, the complexity of the {\sc Outerthickness} problem remains open.

\section{Preliminaries}
\label{sec:prelim}

We may, without loss of generality, restrict to only simple graphs in the whole paper.
This is since, in each of the formulated problems, whenever an edge $e$ is a part of an uncrossed subdrawing (as discussed next),
any other edge parallel to $e$ can be drawn uncrossed closely along $e$, too.

Let $D'$ be a subdrawing of $D$ consisting of only uncrossed edges of $D$. Note that we do not require $D'$ to contain all such edges.
In this situation, we call $D'$ an \emph{uncrossed subdrawing} of $G$ and we say that it \emph{represents} the subgraph of $G$ formed by edges that are drawn in $D'$.  Formally, $D'$ is an uncrossed subdrawing of $G$ if there exists a drawing $D$ of a graph $G$ such that $D'$ is induced by a subset of the uncrossed edges of $D$.

\begin{lemma}
\label{lem-basicProperties}
Let $D'$ be an uncrossed subdrawing of a connected graph $G$.
Then $D'$ is a planar drawing and, for every edge $\{u,v\}$ of $G$, the vertices $u$ and $v$ are incident to a common face of~$D'$.
Moreover, there is an uncrossed subdrawing $D''$ of $G$ such that $D''$ represents a connected supergraph of the graph represented by $D'$.
\end{lemma}
\begin{proof}
The drawing $D'$ is clearly planar, as, by the definition of $D'$, each edge of $D'$ is uncrossed in a drawing $D$ of $G$ and thus also in $D'$.
Moreover, it is a folklore fact that two vertices $u$ and $v$ in a planar drawing, here in $D'$, are {\em not} incident to a common face if and only if there exists a cycle $C\subseteq D'$ such that $u$ and $v$ are drawn on different sides of~$C$.
In the latter case, however, the edge $\{u,v\}$ would cross some edge of $C$ in $D$, which is impossible since no edge of $D'$ is crossed.

We prove the second part by induction on the number of connected components represented by $D'$.
The case of one component is trivial, as $D''=D'$.
Otherwise, since $G$ is connected, there exists an edge $e=\{u,v\}$ of $G$ that is not drawn in $D'$ and such that $u$ and $v$ belong to different components represented by $D'$.
By the first part of the lemma, the vertices $u$ and $v$ are incident to the same face of $D'$.
So, let $D^+$ be the planar drawing obtained from $D'$ by adding a crossing-free arc representing the edge $e$.
Clearly, every edge of $G$ is still incident to a common face of $D^+$, and so $D^+$ can be completed into a drawing of $G$ such that $D^+$ stays uncrossed.
The subgraph of $G$ represented by $D^+$ has fewer components than we started with, and so we find the desired $D''$ by induction.
\end{proof}

\begin {figure}[ht]
  \centering
\begin{tikzpicture}[scale=1]\small
\coordinate (uu) at (0,0);
\tikzstyle{every node}=[draw, color=black, shape=circle, inner sep=1.3pt, fill=white]
\foreach \u in {0,24,...,360} { \node[thick] (W\u) at (\u:2) {}; }
\node[thick] at (uu) {};
\tikzstyle{every path}=[draw, color=black, semithick]
\coordinate (W384) at (W24);
\begin{scope}[on background layer]
\foreach \u in {0,24,...,360} {   
        \draw (0,0)--(W\u);  \draw (uu)--(W\u);  \coordinate (uu) at (W\u);
}
\end{scope}[on background layer]
\end{tikzpicture}
  \caption {The wheel graph $W_{16}$.}
  \label{fig-W16}
\end {figure}
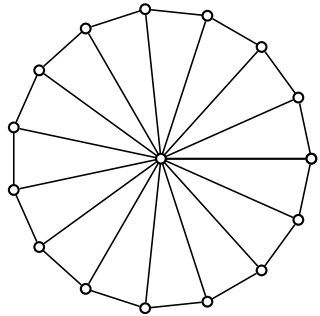

For a graph $G$, let $h(G)$ be the maximum number of uncrossed edges in some drawing $D$ of $G$. 
Let $DW_n$ be a planar drawing of the wheel graph $W_n$ on $n$ vertices; see~\Cref{fig-W16}.
Note that $DW_n$ is unique up to homeomorphism of the sphere and reflection as $W_n$ is 3-connected.
The following result by Ringel~\cite{ringel64} gives a formula for $h(K_n)$ for every integer $n \geq 4$,
and additionally claims that drawings of $K_n$ with the maximum number of uncrossed edges have a unique structure. 

\begin{theorem}[\cite{ringel64}]
\label{thm-ringel64}
For every integer $n \geq 4$, we have
$h(K_n) = 2n-2$.
Moreover, if $D$ is a drawing of $K_n$ with $2n-2$ uncrossed edges, then $D$ contains the drawing $DW_n$ with all edges from $D \setminus DW_n$ being drawn in the outer face of $DW_n$. 
\end{theorem}

We also mention an analogous result for the complete bipartite graphs $K_{m,n}$, derived by Mengersen~\cite{germanPaper}.

\begin{theorem}[\cite{germanPaper}]
\label{thm-germanPaper}
For all positive integers $m$ and $n$ with $m \leq n$, we have
\[
h(K_{m,n}) = 
\begin{cases}
    2m+n-2, & \text{for } m=n \\
    2m+n-1, & \text{for } m < n < 2m \\
    2m+n, & \text{for } 2m \leq n.
\end{cases}
\]
\end{theorem}

We are going to use the parameter $h(G)$ to estimate the uncrossed number of $G$.
Let $\{D_1,\dots,D_k\}$ be an uncrossed collection of drawings of a graph $G$ that has $m$ edges.
Since every drawing $D_i$ contains at most $h(G)$ edges that are uncrossed by any other edge in $D_i$, we immediately obtain the following lower bound
\begin{equation}
\label{eq-boundH}
\Ucr(G) \geq  \left\lceil \frac{m}{h(G)}\right\rceil.
\end{equation}
This bound together with \Cref{thm-ringel64,thm-germanPaper} give us quite close estimates for $\Ucr(K_n)$ and $\Ucr(K_{m,n})$, respectively.
In particular, for $n\geq 2$ we have
\begin{equation}
\label{eq-KnLowerBound}
\Ucr(K_n) \geq  \left\lceil \frac{\binom{n}{2}}{2n-2}\right\rceil.
\end{equation}
On the other hand, we recall the upper bound~\eqref{eq-outerthickness} on the uncrossed number of an arbitrary graph $G$ using the notion of outerthickness of $G$.

\section{Proof of \texorpdfstring{\Cref{thm:complete}}{Theorem 1}}

In this section, we prove \Cref{thm:complete} by providing the exact formula for the uncrossed number of complete graphs.
That is, we show
\[
\Ucr(K_n) = 
\begin{cases}
    \lceil \frac{n+1}{4} \rceil, & \text{for } n \notin \{4,7\} \\
    3, & \text{for } n=7\\
    1, & \text{for } n=4
\end{cases}
\]
for every positive integer $n$.

We start with the upper bound, which is easier to prove.
For $n \notin \{4,7\}$, the upper bound follows from~\eqref{eq-outerthickness} and~\eqref{eq-outerthicknessKn} as we have 
\[\Ucr(K_n) \leq \theta_o(K_n) = \left\lceil \frac{n+1}{4}\right\rceil.\]
For $n=4$, we obviously have $\Ucr(K_4)=1$ as $K_4$ is planar.
Finally, $\Ucr(K_7)=3$ was proved by Hlin\v{e}n\'{y} and Masa\v{r}\'{i}k~\cite[Proposition 3.1]{masHlin23}.

It remains to prove the lower bound.
Since we already know that $\Ucr(K_7)=3$ and $\Ucr(K_4)=1$ and the statement is trivial for $n \leq 3$, it suffices to consider the case $n \geq 5$ with $n \neq 7$.
Let $\{D_1,\dots,D_k\}$ be an uncrossed collection of drawings of $K_n$ and let $D'_1,\dots,D'_k$ be corresponding uncrossed subdrawings of $K_n$ such that $D'_1\cup\dots\cup D'_k$ covers $E(K_n)$. 
By~\eqref{eq-KnLowerBound}, 
\[\Ucr(K_n) \geq \left\lceil \frac{\binom{n}{2}}{2n-2}\right\rceil.\]
By \Cref{thm-ringel64}, we get that if any uncrossed subdrawing $D'_i$ contains $2n-2$ edges, then $D'_i$ (as a wheel) contains a \emph{universal vertex}, that is, a vertex that is adjacent to all remaining vertices in $D'_i$.
If every drawing $D'_i$ contains at most $2n-3$ edges, then 
\begin{equation}
\label{eq-desiredComplete}
\Ucr(K_n) \geq \left\lceil \frac{\binom{n}{2}}{2n-3}\right\rceil = \left\lceil \frac{n}{4} + \frac{n}{4(2n-3)}\right\rceil = \left\lceil \frac{n+1}{4}\right\rceil
\end{equation}
and we are done.

Thus, suppose that some drawing $D'_i$ contains $2n-2$ edges.
Without loss of generality, we can assume $i=1$.
We then know that $D'_1$ contains a universal vertex $v$.
In every drawing $D'_j$ with $j > 1$, the edges incident to $v$ are already counted for $D'_1$, thus we can consider the drawings $D'_2,\dots,D'_k$ to be uncrossed drawings for $K_{n-1}$ obtained from $K_n$ by removing $v$.
Note that this form an uncrossed collection of drawings of $K_{n-1}$ together with the uncrossed edges of $D'_1$ not incident to $v$ and hence their uncrossed edges cover $E(K_{n-1})$.
Then, each $D'_j$, for $j\in[2,k]$ contributes at most $2n-4$ new uncrossed edges of $K_{n-1}$ as $h(K_{n-1}) = 2(n-1)-2=2n-4$ by \Cref{thm-ringel64}. 
So the number $k$ of drawings satisfies
\begin{equation}
  \label{eq-final}
\binom{n}{2} \leq 2n-2 + (k-1)(2n-4) = (2n-4)k + 2.
\end{equation}
However, 
$(2n-4)k+2\le (2n-3)k$ when $k\ge 2$, which is satisfied for $n \geq 5$ by \eqref{eq-final}.
As $\Ucr(K_n)$ is the smallest  integer $k$ that satisfies $\binom{n}{2} \leq k(2n-3)$, we again obtain Inequality~\eqref{eq-desiredComplete}.
\mbox{}\hfill$\qed$

\section{Proof of \texorpdfstring{\Cref{thm:completeBip}}{Theorem 2}}
\label{sec:completeBip}

In this section, we prove \Cref{thm:completeBip} by providing the exact formula for the uncrossed number of complete bipartite graphs.
That is, we show
\[
   \Ucr(K_{m,n})
= 
  \begin{cases}
    \lceil\frac{mn}{2m+n-2}\rceil, & \text{for } 3 \leq m \leq n \leq 2m-2 \\
    \lceil\frac{mn}{2m+n-1}\rceil, & \text{for } n=2m-1 \\
    \lceil\frac{mn}{2m+n}\rceil, & \text{for } 6 \leq 2m \leq n \\
    1, & \text{for } m \leq 2
  \end{cases}
\]
for all positive integers $m$ and $n$  with $m \leq  n$.
We start with the following useful lemma.

\begin{lemma}%
\label{lem-maxOuterplanar}
For positive integers $m$ and $n$ with $m \leq n$, every outerplanar graph $G$ that is a subgraph of $K_{m,n}$ has at most $2m+n-2$ edges.
\end{lemma}

\begin{proof}
Let $G$ be an outerplanar graph that is a subgraph of $K_{m,n}$ and let $v$ and $e$ be the number of vertices and edges of $G$, respectively. 
Since we are interested in the maximum number of edges of $G$, we can assume without loss of generality that $G$ is connected and spanning in $K_{m,n}$.
Let $f$ be the number of faces in an outerplanar drawing $D$ of~$G$.
By Euler's formula, we have $v-e+f=2$.
We also have $4(f-1)+2n \leq 2e$, because if we traverse the boundaries of all faces of $D$, then each edge is visited exactly twice. Moreover, each face, except for the outer face, gives a contribution of at least 4 edges.
For the outer face, the contribution is at least $2n$ since $G$ as a subgraph of $K_{m,n}$ is bipartite, all $n$ vertices of the larger color class must be traversed, and the vertices of the outer face have to alternate between the color classes.

By subtracting Euler's formula $2$ times from half of the inequality, we obtain 
\[2(f-1)+n -2v+2e-2f \leq e - 4,\]
which gives
\[e \leq 2v-n-2.\]
Since, $v=m+n$, we obtain $e \leq 2m+n-2$.
\end{proof}

We note that the bound from \Cref{lem-maxOuterplanar} is tight for all values of $m$ and $n$.
The bound is attained for the ladder graph on $2m$ vertices with $3m-2$ edges where we suitably append $n-m$ leaves, obtaining an outerplanar subgraph of $K_{m,n}$ with $m+n$ vertices and $2m+n-2$ edges; see Figure~\ref{fig-ladder}.

\begin {figure}[ht]
  \centering
  \includegraphics{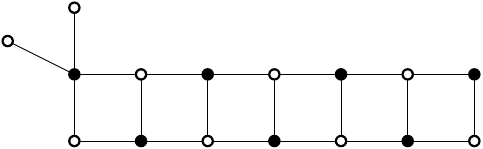}
  \caption {An example of a ladder graph with leaves that is an outerplanar subgraph of $K_{m,n}$ with $2m+n-2$ edges. Here, we have $m=7$ black vertices and $n=9$ white vertices.}
  \label{fig-ladder}
\end {figure}

To obtain tight formulas for $\Ucr(K_{m,n})$, we will need to also consider graphs that are not outerplanar.
A \emph{double cycle} is a graph that is obtained from a cycle (where we also allow cycle on 2 vertices) by replacing each edge $\{u,v\}$ with a copy of $C_4$ on vertices $\{u,v,x,y\}$ and edges $\{u,x\}$, $\{u,y\}$, $\{v,x\}$, and $\{v,y\}$ where $x$ and $y$ are new vertices; see \Cref{fig:basic-dc}.
We call the vertices from the original cycle ($u$ and $v$) \emph{black vertices} and the new vertices ($x$ and $y$) \emph{white vertices}.
A \emph{double cycle with leaves} is a double cycle where we attach new vertices of degree 1 to black vertices; see \Cref{fig:dc-with-leaves}.
These new vertices are also white.
Observe that a double cycle with leaves is a bipartite graph where the two color classes are formed by black and white vertices, respectively.

\begin {figure}[ht]
\centering
\subcaptionbox{\centering\label{fig:basic-dc}}[.49\textwidth]{\includegraphics[page=1]{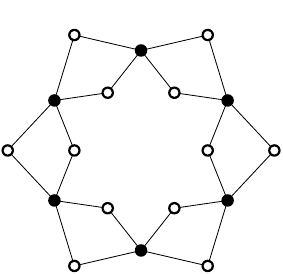}}
\subcaptionbox{\centering\label{fig:dc-with-leaves}}[.49\textwidth]{\includegraphics[page=2]{fig-doubleCycle.pdf}}
\caption {(a) An example of a double cycle. (b) An example of a double cycle with leaves.}
\label{fig-doubleCycle}
\end {figure}

We start with two technical lemmas which will be used to prove the lower bounds.
For a graph $G$ containing a cycle $C$, a path $P$ of length two in $G$ is called a \emph{double-chord of $C$} if both ends of $P$ lie on $C$, but the middle vertex of $P$ is not on $C$.

\begin{lemma}%
\label{lem:bwchords}
Let $G$ be a bipartite connected graph whose color classes are called black and white, $C$ be a cycle of length at least $6$ in $G$, and $P_1$ and $P_2$ be two paths of $G$ satisfying all of the following assumptions:
\begin{enumerate}[a)]\parskip0pt
  \item \label{chords:a} For $i=1,2$, $P_i$ is a chord of $C$ (that is, an edge with black and white ends on $C$) or a double-chord of $C$ with black ends.
  \item \label{chords:b} If both $P_1$ and $P_2$ are double-chords with ends at distance two on $C$, then they both end in the same pair of vertices on $C$.
  \item \label{chords:c} If both $P_1$ and $P_2$ are disjoint double-chords with ends at distance four on $C$, then the length of $C$ is not equal to $8$.
  \item \label{chords:d} For $i=1,2$, denoting by $x$ and $y$ the ends of $P_i$ on $C$, both subpaths of $C$ with the ends $x$ and $y$ contain an internal vertex not incident to $P_{3-i}$.
\end{enumerate}
Assume that $D$ is a planar drawing of $G$, and that if both $P_1$ and $P_2$ are chords, then they are drawn in different regions of the subdrawing of $C$ in $D$.

Then there is a pair of vertices of $G$ of opposite colors which do not share a face of~$D$.
\end{lemma}

\begin{proof}
Let the ends of $P_1$ and $P_2$ on $C$ be $x_1,x_2,x_3,x_4$ in this clockwise cyclic order on $C$ (regardless of whether they belong to $P_1$ or $P_2$), and note that we do not require these four vertices to be pairwise distinct.
Let $Q_i$ denote the subpath of $C$ from $x_i$ to $x_{i+1}$ (where $x_5=x_1$) in the clockwise direction.
We consider first the case that one of the paths joins two consecutive vertices of $x_1,x_2,x_3,x_4$ in their cyclic order on $C$, say, $P_1$ with the ends $x_1$ and $x_2$ (same as $Q_1$). 
Then $P_2$ symmetrically has ends $x_3$ and $x_4$ (same as $Q_3$), and it may possibly happen that $x_2=x_3$ and/or $x_4=x_1$.

If both $P_1$ and $P_2$ are chords, then, from our specific assumption on the drawing $D$, the internal vertices of $Q_1$ are separated from the internal vertices of $Q_3$ (by the cycle of $C\cup P_1\cup P_2$ passing through $P_1$ and $P_2$), and since they include both black and white vertices, we get a pair of opposite colors with no common face in $D$.
So, up to symmetry, $P_1$ is a double-chord.

First, we handle the case that $P_{2}$ is a chord.
Observe that there has to exist a white vertex on both $P_1$ and $Q_1$.
One of those two vertices, denoted as $y$, is separated from $Q_3$ by drawing of $P_2\cup P_1\cup Q_1$.
However, as $P_2$ is a chord, $Q_3$ contains at least 2 internal vertices and hence, at least one of them is black.

Second, we assume $P_2$ is a double-chord.
By symmetry again, we may as well assume that $Q_1$ is not longer than $Q_3$ if $P_2$ is also a double-chord.
Then, using assumption~\eqref{chords:b} if $P_2$ is a double-chord, we have that $Q_3$ is of length at least $3$.
Now, one of $P_1$ or $Q_1$ has an internal vertex $y$ separated in $D$ from the internal vertices of $Q_3$, and we can thus choose an internal vertex of $Q_3$ of color opposite to that of $y$ and not sharing a face with $y$.

Thus, it remains to consider that, up to symmetry, $P_1$ has ends $x_1$ and $x_3$ and $P_2$ has ends $x_2$ and $x_4$, such that all $x_1,x_2,x_3,x_4$ are distinct.
In the drawing $D$, clearly, no internal vertex of $Q_1$ shares a common face with an internal vertex of $Q_3$, and likewise for $Q_2$ and $Q_4$. 
It now suffices to find such a vertex pair with opposite colors.

If $C$ is of length $6$, then the only case fulfilling the assumptions is that of $P_1$ and $P_2$ both being chords, and then the remaining two vertices on $C$ indeed are of opposite colors and form the desired pair.
If $C$ is of length at least $8$, then, respecting assumption~\eqref{chords:c}, one of the paths, say $Q_1$, is of length at least $3$ and thus contains internal vertices of both colors.
If $Q_3$ is of length at least $2$, then we again get a desired pair of opposite colors. So, let $Q_3$ be a single edge of $C$.
If any one of $Q_2$ or $Q_4$ is also a single edge, we contradict assumption~\eqref{chords:d}.
Then both $Q_2$ and $Q_4$ are of length at least $2$, and the neighbor of $x_3$ on $Q_2$ is of the opposite color to the neighbor of $x_4$ on $Q_4$, and we are again finished.
\end{proof}

We now prove the following lemma, which describes the structure of an uncrossed drawing of $K_{m,n}$ with $3 \leq m < n$ if it contains many edges.
In fact, the statement of this result is stronger than what we need for the proof of \Cref{thm:completeBip}.

\begin{lemma}%
\label{lem-compBipStructure}
For integers $m$ and $n$ with $3 \leq m < n$ and for $a \in \{0,1\}$, let $D'$ be an uncrossed subdrawing of $K_{m,n}$ with at least $2m+n-a$ edges.
Let the color classes of $K_{m,n}$ be called black and white such that we have $m$ black vertices.

Then either $D'$ contains a vertex that is adjacent to all black vertices, or $n \geq 2m-a$ and $D'$ represents a subgraph of a double cycle with leaves.
\end{lemma}

\begin{proof}
We assume that no white vertex is adjacent to all black vertices as otherwise we are done.
The uncrossed subdrawing $D'$ is planar by definition and, by \Cref{lem-basicProperties},
for every black vertex $b$ and every white vertex $w$, there is a face of $D'$ which is incident with vertices $b$ and~$w$.

Since $D'$ is a drawing of a graph $G'$ that is a subgraph of $K_{m,n}$ and has $2m+n-1> 2m+n-2$ edges, it follows from \Cref{lem-maxOuterplanar} that this graph is not outerplanar.
Thus, $G'$ contains a $K_{2,3}$ or $K_4$ as a minor. However, if $K_4$ is a minor of a bipartite graph $G'$, then $G'$ also contains a minor resulting from $K_4$ by subdividing at least one edge, which in turn contains a $K_{2,3}$ minor, too.
Since $K_{2,3}$ contains two 4-cycles that share a path on three vertices, it follows that there are cycles $C$ and $C'$ in $G'$ that share exactly one path $P$ that contains at least three vertices and we also have $|V(C) \setminus V(C')| \geq 1$ and $|V(C') \setminus V(C)| \geq 1$.

Let $Q$ be the set of vertices of $G'$ that are not in $V(C) \cup V(C')$.
Now, we show that all vertices from $Q$ are white.
Since $m < n$, there are more white vertices than black ones and thus it suffices to show that there are no vertices of different colors in~$Q$.
Suppose for contradiction that there are black and white vertices in $Q$.
The cycles $C$ and $C'$ determine in $D'$ three pairwise disjoint open regions $\varrho_1$, $\varrho_2$, and $\varrho_3$ in $\mathbb{R}^2$.
If any region $\varrho_i$ contains a black vertex, then, since $m<n$ and since any two vertices of different colors share a face of $D'$, $\varrho_i$ has to contain all white vertices from $Q$.
In particular, $\varrho_i$ contains at least one black and at least one white vertex.
This, however, is impossible as then there is a vertex contained in one of the sets $P$, $V(C) \setminus V(C')$, or $V(C') \setminus V(C)$ that does not share a face of $D'$ with any of these two vertices in $\varrho_i$ (here we used that $P$ has at least 3 vertices and that $|V(C')\setminus V(C)| \geq 1$ and $|V(C)\setminus V(C')| \geq 1$).
It follows that each region $\varrho_i$ contains only white vertices and thus all vertices in $Q$ are white.

Now, we show that all black vertices lie on a single cycle in $G'$.
Since all vertices from $Q$ are white, all black vertices lie in $V(C) \cup V(C')$.
Let $P_1$, $P_2$, and $P_3$ be the three paths that we obtain from $C$ and $C'$ by removing their two common vertices $u$ and $v$ of degree 3.
If each $P_i$ contains a black vertex, then $Q$ is empty as every white vertex in each of the regions $\varrho_1$, $\varrho_2$, and $\varrho_3$ does not share a face of $D'$ with some of these black vertices.
Since $m<n$, each path $P_i$ contains more white vertices than black, and then $u$ and $v$ are black.
Therefore, each path $P_i$ has at least three vertices.
Indeed, $P_1$ having only one white vertex already proves the claim as it does not contain any black verteces.
Since $Q$ is empty and $G'$ has only at most $m+n+1< 2m+n-1$ edges as $m \geq 3$, there is an edge of $G'$ that is not an edge of $C$ nor $C'$.
Then, however, there is a white and black vertex that do not share a face in $D'$ as can be shown by routine case analysis. 
This contradicts the assumptions and thus some $P_i$ does not contain a black vertex.
Therefore, all black vertices lie on some cycle $C''$ of $G'$.
We assume that $C''$ is chosen such that the number of edges of $G'$ which are chords of $C''$ (have both ends on $C''$) is minimized.

Let $R$ be the set of vertices of $G'$ that are not contained in $C''$.
All vertices from $R$ are white and all their neighbors are black and lie on $C''$. 
We will be finished with the main part if we show that every vertex of $R$ has at most one neighbor on $C''$, or exactly two neighbors on $C''$ which are at distance $2$ along $C''$, and that there are no chords of $C''$ in $G'$.
Indeed, then $G'$ would be a subgraph of a double cycle with leaves, or there would be two vertices $r_1,r_2\in R$ having the same pair of neighbors on $C''$ which is a contradiction via \Cref{lem:bwchords}.

Assume that some vertex $r\in R$ has two neighbors $x$ and $y$ at a distance at least $4$ along~$C''$, and recall that $r$ is not adjacent to all black vertices.
If there were no chords of $C''$ in $G'$ and every vertex of $R\setminus\{r\}$ had less than $2$ neighbors on $C''$, then $G'$ would have at most $2m+n-2$ edges ($2m$ on $C''$ plus at most $m-1$ incident to $r$ plus $n-m-1$ incident to the remaining vertices of $R$), a contradiction.
Therefore, there exists a chord of $C''$ or a vertex $r'\in R\setminus\{r\}$ with two neighbors on $C''$ to which, together with the double-chord of $r$, we may apply \Cref{lem:bwchords} to derive a contradiction to the assumptions; except in one specific case of $C''$ of length $8$ which has two ``crossing'' double-chords in $G'$.
In the latter exceptional case, we have $m=4$ and $n=6$, and since we assume at least $2m+n-1=13$ edges, one of $r$ or $r'$ has another neighbor on $C''$, which by a simple analysis again leads to a pair of vertices of opposite colors not sharing any face in $D'$.

Hence, no vertex $r\in R$ has two neighbors at a distance at least $4$ along~$C''$.
This also implies that no vertex $r\in R$ has three or more neighbors in $C''$, since either $r$ would be adjacent to all three black vertices of $C''$, or $r$ would have two neighbors at a distance at least $4$ along $C''$.
We are left with showing that there are no chords of $C''$ in $D'$.
First, if all chords were drawn in $D'$ in the same region of $C''$ and no vertex of $R$ had two neighbors on $C''$, then $G'$ would again have at most $2m+n-2$ edges since $G'$ is bipartite (by \Cref{lem-maxOuterplanar}), a contradiction to the assumptions.
Otherwise, we have two chords $P_1=\{f\}$ and $P_2=\{f'\}$ drawn in different regions of $C''$ in $D'$, or a chord $P_1=\{f\}$ of $C''$ and a double-chord $P_2$ formed by $r\in R$ with two neighbors on $C''$.
We again get a contradiction to the assumptions by applying \Cref{lem:bwchords} to $P_1$ and $P_2$, except in one case of $r$ having both its neighbors adjacent to one end of $f$ on $C''$. 
In the latter case, we ``reroute'' the cycle $C''$ through $r$, and so decrease the number of chords, which contradicts our choice of $C''$.

Since all vertices in $R$ have degree at most 2, the sum of their degrees is at most the number of edges that are not in $C''$, that is, $2m+n-a-2m=n-a$.
Thus, $2|R| \geq n-a$.
Since $|R| = (n+m) - 2m = n-m$, we obtain $2(n-m) \geq n-a$, which can be rewritten as $n \geq  2m-a$.
This finishes the proof.
\end{proof}

Subsequently, we turn our attention to the upper bounds for the proof of \Cref{thm:completeBip}.

\begin{lemma}
\label{lem-2m}
For every $n \geq 2m$ and $m \geq 3$ the edges of the complete bipartite graph $K_{m,n}$ can be covered by $\left\lceil \frac{mn}{2m+n} \right\rceil$ double cycles with leaves (not necessarily edge-disjoint).
\end{lemma}

\begin{proof}
Let the vertices of the partition class of size $n$ be $w_{1}, w_{2}, \ldots, w_{n}$ and let the other $m$ vertices be $b_{1}, b_{2}, \ldots, b_{m}$. Because $n \geq 2m$ we can build double cycles with $n-2m$ leaves (\Cref{fig-doubleCycle}) in which all $b_{i}$ are the black vertices and all $w_{j}$ are the white vertices. For this, we consider the vertices $b_{i}$ and $w_{j}$, respectively, circularly ordered by their indices, as illustrated in \Cref{fig-doubleCycleCover}. 
Each double cycle with leaves in our cover will be defined by a sequence of (black) degrees $(d_{1}, d_{2}, \ldots, d_{m})$ with $d_{i} \geq 4$ and $\sum d_{i} = 2m+n$, and a starting index $1 \leq s \leq n$ in the following way.

The black vertices $b_{i}$ occur in the given cyclic order on a double cycle, and for all $1 \leq i \leq m$ there are exactly $d_{i}-4$ white leaves to the vertex $b_{i}$. 
The white vertices $w_{j}$ are assigned as neighbors to the black vertices on that double cycle with leaves such that vertex $b_{i}$ is adjacent to vertices $w_{j}$ with $s_{i} \leq j < s_{i}+d_{i}$ where $s_{i} = s + \sum_{k<i} (d_{k} - 2)$ and $j$ is considered modulo $n$. In other words, we distribute the white neighbors according to their circular order such that vertex $b_{1}$ is adjacent to a consecutive block of $d_1$ white vertices starting at vertex $w_{s}$, and the blocks for the other black vertices are shifted such that each has a two-vertex overlap with the previous block.

Assume we have constructed one such double cycle with leaves given by the degree sequence $(d_{1}, d_{2}, \ldots, d_{m})$ and the starting index $s$. Then we construct all remaining of the sought $\left\lceil \frac{mn}{2m+n} \right\rceil$ double cycles in the following way. We successively increase the starting index by the degree of the first vertex $s \rightarrow s+d_{1}$ and shift the degree sequence one to the left $(d_{1}, d_{2}, \ldots, d_{m}) \rightarrow (d_{2}, d_{3}, \ldots, d_{m}, d_{1})$. 
This ensures that in the next double cycle every black vertex $b_{i}$ is adjacent to a block of white vertices starting at the first white vertex in circular order that $b_{i}$ was not yet adjacent to. Hence, in the union of the constructed double cycles, every black vertex has a continuous interval of adjacent white neighbors. It remains to construct a starting sequence $(d_{1}, d_{2}, \ldots, d_{m})$ such that the sum of every $\left\lceil \frac{mn}{2m+n} \right\rceil$ circularly consecutive degrees is at least $n$, which we do in the following.

Let $\ell = \left\lceil \frac{mn}{2m+n} \right\rceil$, which is the number of constructed double cycles, and $a = \frac{2m+n}{m}$, which is the average degree in the degree sequence for the black vertices. Set the degrees to $d_{i} = \left\lfloor i \cdot a \right\rfloor - \sum_{k<i} d_{k}$, which are integers and we have $d_{i}\geq4$. Then the sum of the first $\ell$ degrees is
\begin{align}
\sum_{k \leq \ell} d_{k} = \left\lfloor \ell \cdot a \right\rfloor = \left\lfloor \left\lceil \frac{mn}{2m+n} \right\rceil \cdot a \right\rfloor \geq \left\lfloor \frac{mn}{2m+n} \cdot a \right\rfloor = \left\lfloor \frac{mn}{2m+n} \cdot \frac{2m+n}{m} \right\rfloor = n \label{eq:sumEll}\end{align} as intended.
So, if we extend the definition of the degrees for indices greater than $m$ we prove
$d_{i+m} = \lfloor(i+m)a\rfloor - \sum_{k<i+m} d_k =  d_{i}$ as $ma= 2m+n$ is an integer and by the first equality in \cref{eq:sumEll}, we have $\sum_{k\leq m} =\lfloor ma \rfloor=ma$.
Therefore, we get that the sum of $\ell$ degrees starting at $d_{i+1}$ in the circular order of degrees is \[\sum_{i < k \leq \ell+i} d_{k} = \left\lfloor (\ell+i) \cdot a \right\rfloor - \left\lfloor i \cdot a \right\rfloor \geq \left\lfloor \left\lfloor \ell \cdot a \right\rfloor + \left\lfloor i \cdot a \right\rfloor \right\rfloor - \left\lfloor i \cdot a \right\rfloor = \left\lfloor \ell \cdot a \right\rfloor \geq n.\] Hence, all circular intervals of length $\ell$ sum up to at least $n$.
\end{proof}

\begin {figure}[t]
\centering
\hbox{\small
\begin{tikzpicture}[scale=0.8] 
\coordinate (uu) at (0,0);  
\tikzstyle{every node}=[draw, color=black, shape=circle, inner sep=1.3pt, fill=black]
\foreach \u in {0,40,...,680} { \node (B\u) at (\u+90:0.8) {}; }
\foreach \u in {0,15,...,705} { \node[thick,fill=white] (W\u) at (\u:2.5) {}; }
\tikzstyle{every path}=[draw, color=black, semithick, bend right=48] 
\tikzmath{integer \globx; \globx = 0; integer \xx; integer \xI; integer \xJ;}
\foreach \d/\a in {45/0,60/40,60/80,45/120,60/160,60/200,45/240,60/280,60/320} {       
        \tikzmath{\xx = \globx; \xI = \globx+15; \xJ = \globx+\d; \globx = \xJ-15;}
        \xdef\globx{\globx}
        \foreach \u in {\xx,\xI,...,\xJ}  \draw (W\u) to (B\a); 
}
\tikzstyle{every node}=[draw, color=black, shape=circle, inner sep=2.5pt]
\draw node at (B0) {} node at (W0) {};
\end{tikzpicture}
\quad
\begin{tikzpicture}[scale=0.8] 
\coordinate (uu) at (0,0);
\tikzstyle{every node}=[draw, color=black, shape=circle, inner sep=1.3pt, fill=black]
\foreach \u in {0,40,...,680} { \node (B\u) at (\u+90:1) {}; }
\foreach \u in {0,15,...,705} { \node[thick,fill=white] (W\u) at (\u:2.5) {}; }
\tikzstyle{every path}=[draw, color=black, semithick, bend left=0]
\tikzmath{integer \globx; \globx = 60; integer \xx; integer \xI; integer \xJ;}
\foreach \d/\a in {60/0,60/40,45/80,60/120,60/160,45/200,60/240,60/280,45/320} {       
        \tikzmath{\xx = \globx; \xI = \globx+15; \xJ = \globx+\d; \globx = \xJ-15;}
        \xdef\globx{\globx}
        \foreach \u in {\xx,\xI,...,\xJ}  \draw (W\u) to (B\a); 
}
\tikzstyle{every node}=[draw, color=black, shape=circle, inner sep=2.5pt]
\draw node[label=below:$b_1$] at (B0) {} node[label=right:$w_1$] at (W0) {};
\end{tikzpicture}
\quad
\begin{tikzpicture}[scale=0.8] 
\coordinate (uu) at (0,0);
\tikzstyle{every node}=[draw, color=black, shape=circle, inner sep=1.3pt, fill=black]
\foreach \u in {0,40,...,680} { \node (B\u) at (\u+90:0.8) {}; }
\foreach \u in {0,15,...,705} { \node[thick,fill=white] (W\u) at (\u:2.5) {}; }
\tikzstyle{every path}=[draw, color=black, semithick, bend left=48]
\tikzmath{integer \globx; \globx = 135; integer \xx; integer \xI; integer \xJ;}
\foreach \d/\a in {60/0,45/40,60/80,60/120,45/160,60/200,60/240,45/280,60/320} {
        \tikzmath{\xx = \globx; \xI = \globx+15; \xJ = \globx+\d; \globx = \xJ-15;}
        \xdef\globx{\globx}
        \foreach \u in {\xx,\xI,...,\xJ}  \draw (W\u) to (B\a); 
}
\tikzstyle{every node}=[draw, color=black, shape=circle, inner sep=2.5pt]
\draw node at (B0) {} node at (W0) {};
\end{tikzpicture}
}\bigskip

\hbox{\small
\begin{tikzpicture}[scale=0.8] 
\coordinate (uu) at (0,0);
\tikzstyle{every node}=[draw, color=black, shape=circle, inner sep=1.3pt, fill=black]
\foreach \u in {0,40,...,680} { \node (B\u) at (\u+90:1.4) {}; }
\foreach \u in {0,15,...,705} { \node[thick,fill=white] (W\u) at (\u:2.5) {}; }
\tikzstyle{every path}=[draw, color=black, semithick, bend left=6]
\tikzmath{integer \globx; \globx = 210; integer \xx; integer \xI; integer \xJ;}
\foreach \d/\a in {45/0,60/40,60/80,45/120,60/160,60/200,45/240,60/280,60/320} {
        \tikzmath{\xx = \globx; \xI = \globx+15; \xJ = \globx+\d; \globx = \xJ-15;}
        \xdef\globx{\globx}
        \foreach \u in {\xx,\xI,...,\xJ}  \draw (W\u) to (B\a); 
}
\tikzstyle{every node}=[draw, color=black, shape=circle, inner sep=2.5pt]
\draw node at (B0) {} node at (W0) {};
\end{tikzpicture}
\quad
\begin{tikzpicture}[scale=0.8] 
\coordinate (uu) at (0,0);
\tikzstyle{every node}=[draw, color=black, shape=circle, inner sep=1.3pt, fill=black]
\foreach \u in {0,40,...,680} { \node (B\u) at (\u+90:1.6) {}; }
\foreach \u in {0,15,...,705} { \node[thick,fill=white] (W\u) at (\u:2.5) {}; }
\tikzstyle{every path}=[draw, color=black, semithick]
\tikzmath{integer \globx; \globx = 270; integer \xx; integer \xI; integer \xJ;}
\foreach \d/\a in {60/0,60/40,45/80,60/120,60/160,45/200,60/240,60/280,45/320} {
        \tikzmath{\xx = \globx; \xI = \globx+15; \xJ = \globx+\d; \globx = \xJ-15;}
        \xdef\globx{\globx}
        \foreach \u in {\xx,\xI,...,\xJ}  \draw (W\u) to (B\a); 
}
\tikzstyle{every node}=[draw, color=black, shape=circle, inner sep=2.5pt]
\draw node at (B0) {} node at (W0) {};
\end{tikzpicture}
\qquad
\begin{tikzpicture}[scale=0.8] 
\coordinate (uu) at (0,0);
\tikzstyle{every node}=[draw, color=black, shape=circle, inner sep=1.3pt, fill=black]
\foreach \u in {0,40,...,680} { \node (B\u) at (\u+90:0.8) {}; }
\foreach \u in {0,15,...,720} { \node[thick,fill=white] (W\u) at (\u:2.5) {}; }
\tikzstyle{every path}=[draw, color=black, semithick, bend right=48]
\tikzmath{integer \globx; \globx = 345; integer \xx; integer \xI; integer \xJ;}
\foreach \d/\a in {60/0,45/40,60/80,60/120,45/160,60/200,60/240,45/280,60/320} {
        \tikzmath{\xx = \globx; \xI = \globx+15; \xJ = \globx+\d; \globx = \xJ-15;}
        \xdef\globx{\globx}
        \foreach \u in {\xx,\xI,...,\xJ}  \draw (W\u) to (B\a); 
}
\tikzstyle{every node}=[draw, color=black, shape=circle, inner sep=2.5pt]
\draw node at (B0) {} node at (W0) {};
\end{tikzpicture}
}

\caption {An example of the cover of $K_{m,n}$ by $\ell$ double cycles with leaves sought in \Cref{lem-2m}. 
    Here we have $m=9$, $n=24$, $\ell=6$, and the chosen starting degree sequence of the black vertices is $(4,5,5,4,5,5,4,5,5)$.
    The indices of the black and white vertices grow counter-clockwise (with $b_1$ and $w_1$ being emphasized).
    Note that some edges of $K_{m,n}$ are covered twice.}
\label{fig-doubleCycleCover}
\end {figure}
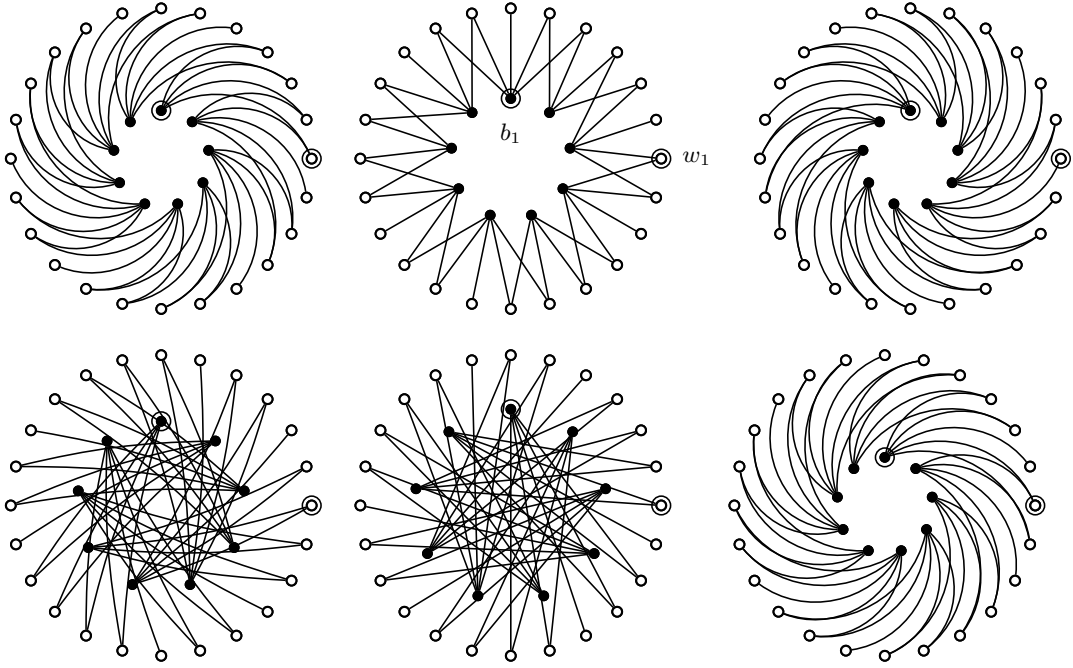

We need one more special case for which we use double cycles with one white vertex removed, that is, all black vertices have degree $4$ except for two consecutive ones that only have degree $3$.

\begin{lemma}%
\label{lem-2m-1}
For $n = 2m-1$ the edges of the complete bipartite graph $K_{m,n}$ can be covered by $\left\lceil \frac{mn}{2m+n-1} \right\rceil$ double cycles with one white vertex removed.
\end{lemma}

\begin{proof}
Since $n = 2m-1$, we aim to create $\left\lceil \frac{mn}{2m+n-1} \right\rceil = \left\lceil \frac{m(2m-1)}{2m+2m-1-1} \right\rceil = \left\lceil \frac{m}{2} \right\rceil$ double cycles with one white vertex removed. In these graphs, every white vertex is adjacent to two consecutive black vertices in the circular order. We start with an arbitrary assignment of the black and white vertices for the first double cycle. For every further double cycle, we keep its basic structure and the assignment of the white vertices, but we cyclically shift the black vertices by two indices. That is, each black vertex is placed in every second position for black vertices in the process. Since we produce $\left\lceil \frac{m}{2} \right\rceil$ double cycles with one white vertex removed in total, every black vertex is adjacent to every white vertex at least once. Hence, each edge of $K_{m,n}$ is covered by at least one of the double cycles.
\end{proof}

We are now ready to determine the uncrossed numbers of complete bipartite graphs $K_{m,n}$ for all positive integers $m$ and $n$.

\begin{proof}[Proof of \Cref{thm:completeBip}]
Let $m$ and $n$ be positive integers with $m \leq n$.
Recall that we want to prove that 
\[
   \Ucr(K_{m,n})
= 
  \begin{cases}
    \lceil\frac{mn}{2m+n-2}\rceil, & \text{for } 3 \leq  m \leq n \leq 2m-2 \\
    \lceil\frac{mn}{2m+n-1}\rceil, & \text{for } n=2m-1 \\
    \lceil\frac{mn}{2m+n}\rceil, & \text{for } 6 \leq 2m \leq n \\
    1, & \text{for } m \leq 2.
  \end{cases}
\]

First, if $m \leq 2$, then $K_{m,n}$ is planar and we have $\Ucr(K_{m,n}) = 1$.
Thus, we assume $m \geq 3$ for the rest of the proof.

We now prove the upper bounds on $\Ucr(K_{m,n})$.
By~\eqref{eq-outerthickness}, we have $\Ucr(K_{m,n}) \leq \theta_o(K_{m,n})$ and, since $\theta_o(K_{m,n}) = \left\lceil \frac{mn}{2m+n-2} \right\rceil$ for all $m$ and $n$ with $m \leq n$ by~\eqref{eq-outerthicknessKmn}, we have $\Ucr(K_{m,n}) \leq \left\lceil \frac{mn}{2m+n-2} \right\rceil$.
This matches our desired formula for $m \leq n \leq 2m-2$, so it suffices to consider $n \geq 2m-1$.
If $n=2m-1$, then the edges of $K_{m,n}$ can be covered by $\left\lceil\frac{mn}{2m+n-1} \right\rceil$ double cycles with one white vertex removed by Lemma~\ref{lem-2m-1}.
If $n \geq 2m$, then the edges of $K_{m,n}$ can be covered by $\left\lceil \frac{mn}{2m+n} \right\rceil$ double cycles with leaves by Lemma~\ref{lem-2m}.
Observe that each double cycle $C$ with leaves can be extended to a drawing of $K_{m,n}$ where the edges of the double cycle are uncrossed as each white vertex shares a face in $C$ with all black vertices.
Altogether, this gives the desired upper bounds on $\Ucr(K_{m,n})$ in all cases.

It remains to verify the lower bounds. Recall that $h(G)$ denotes the maximum number of uncrossed edges in any drawing of $G$.
We get $\Ucr(K_{m,n}) \geq  \left\lceil \frac{mn}{h(K_{m,n})}\right\rceil$ from~\eqref{eq-boundH}. Since
\[
h(K_{m,n}) = 
\begin{cases}
    2m+n-2, & \text{for } m=n \\
    2m+n-1, & \text{for } 2m > n > m \\
    2m+n, & \text{for } n \geq 2m,
\end{cases}
\]
by Theorem~\ref{thm-germanPaper}, we obtain the lower bound
\[
\Ucr(K_{m,n}) \geq  
\begin{cases}
    \left\lceil \frac{mn}{2m+n-2}\right\rceil, & \text{for } m=n \\
    \left\lceil \frac{mn}{2m+n-1}\right\rceil, & \text{for } 2m > n > m \\
    \left\lceil \frac{mn}{2m+n}\right\rceil, & \text{for } n \geq 2m.
\end{cases}
\]
This matches our desired formula in all cases except for $m<n\leq 2m-2$.
Thus, we assume that $m+1 \leq n\leq 2m-2$ from now on.
We want to show that $\Ucr(K_{m,n}) \geq \left\lceil \frac{mn}{2m+n-2}\right\rceil$.
Let $\{D_1,\dots,D_k\}$ be an uncrossed collection of drawings of~$K_{m,n}$ and let $D'_1,\dots,D'_k$ be the corresponding uncrossed subdrawings of $K_{m,n}$.
If every drawing $D'_i$ contains at most $2m+n-2$ edges, then $k \geq \left\lceil \frac{mn}{2m+n-2}\right\rceil$.

Thus, suppose that some $D'_i$ contains at least $2m+n-1$ edges.
Without loss of generality, assume that it is the case for $D'_1$.
Then, $D'_1$ contains exactly $2m+n-1$ edges as \[h(K_{m,n}) \leq 2m+n-1\text{ for }m<n \leq 2m-2.\]
We have $3 \leq m<n$ and thus we can apply Lemma~\ref{lem-compBipStructure}.
Since $n\leq 2m-2$, Lemma~\ref{lem-compBipStructure} implies that $D'_1$  contains a vertex $v$ that is adjacent to all $m$ black vertices.
In every drawing $D'_j$ with $j > 1$, the edges incident to $v$ are already counted for $D'_1$, thus we can consider the drawings $D'_2,\dots,D'_k$ together with the uncrossed edges not incident with $v$ in $D'_1$ to be an uncrossed collection of subdrawings for $K_{m,n-1}$ obtained from $K_{m,n}$ by removing $v$.
Then, each $D'_j$ for $j\in[2,k]$ contributes at most $h(K_{m,n-1})$ new uncrossed edges of~$K_{m,n-1}$.

If each drawing from $D'_2,\dots,D'_k$ contributes at most $2m+n-3$ edges of $K_{m,n-1}$, then 
\[mn \leq (2m+n-1) + (k-1)(2m+n-3) = k(2m+n-3) + 2\]
as every edge of $K_{m,n}$ is uncrossed in some drawing $D'_i$ with $i \in\{1,\dots,k\}$.
This can be rewritten as \begin{align}
  k \geq \frac{mn-2}{2m+n-3}.\label{eq-k}
\end{align}
We have $\frac{mn-2}{2m+n-3} \geq \frac{mn}{2m+n-2}$ for $m \geq 4$ and $n \geq m+2$, and for $m \geq 5$ and $n=m+1$, as the inequality can be equivalently expressed as $n \geq (4m-4)/(m-2)$.
Then, we get the desired bound $k \geq \left\lceil \frac{mn}{2m+n-2}\right\rceil$ as
\begin{equation}
\label{eq-boundingK}
k \geq \frac{mn-2}{2m+n-3} \geq \frac{mn}{2m+n-2}
\end{equation}
and $k$ is an integer.

Recall that we want to prove the lower bound $k \geq \left\lceil \frac{mn}{2m+n-2}\right\rceil$ in each case $m +1 \leq n \leq 2m-2$ where we know that $D'_1$ contains $2m+n-1$ uncrossed edges.

If $n=m+1$, then Theorem~\ref{thm-germanPaper} gives $h(K_{m,n-1}) =2m+(n-1)-2=2m+n-3$ and we see that each drawing from $D'_2,\dots,D'_k$ contributes at most $2m+n-3$ edges of $K_{m,n-1}$.
We are then done by~\eqref{eq-boundingK} if $m \geq 5$.
Otherwise, we have $m=3,n=4$ or $m=4,n=5$.
In both of these cases, the bound (Inequality (\ref{eq-k})) $k \geq \frac{mn-2}{2m+n-3}$ implies that $k \geq 2$, which also matches the upper bound on $\Ucr(K_{m,n})$.

Finally, assume $n \geq m+2$. 
Since we also have $n \leq 2m-2$, we obtain $m \geq 4$.
It follows from Theorem~\ref{thm-germanPaper} that $h(K_{m,n-1}) =2m+(n-1)-1=2m+n-2$ and we get 
\[mn \leq (2m+n-1) + (k-1)(2m+n-2) = k(2m+n-2) + 1, \]
which can be rewritten as $k \geq (mn-1)/(2m+n-2)$. 
Since $k$ is an integer, this implies the desired bound $k \geq \left\lceil \frac{mn}{2m+n-2}\right\rceil$ if $mn$ is not congruent to $1$ modulo $2m+n-2$.
Thus, we assume that $mn = a(2m+n-2) + 1$ for some positive integer $a$.
If every drawing from $D'_2,\dots,D'_k$ contains at most $2m+n-3$ uncrossed edges of $K_{m,n-1}$, then we are done by~\eqref{eq-boundingK} as $m \geq 4$ and $n \geq m+2$.
Otherwise, some drawing from $D'_2,\dots,D'_k$ contains at least $2m+n-2$ uncrossed edges of $K_{m,n-1}$.
Without loss of generality, this happens for $D'_2$.
Since $n \geq m+2$, we obtain $n-1 \geq m+1$ while we are in the case where $n\le 2m-2$ (so $n-1\le 2m-3$). Hence, Lemma~\ref{lem-compBipStructure} implies that $D'_2$  contains a vertex $v'$ that is adjacent to all $m$ black vertices.
Hence, we can conclude the proof, as at least one of the edges incident with $v$ was already uncrossed in $D'_1$. Which in turn (using Theorem~\ref{thm-germanPaper}) implies that each drawing $D'_i$ for $i\in[2,k]$ contains at most $2m+n-3$ uncrossed edges.
\end{proof}

\section{Proof of \texorpdfstring{\Cref{thm-lowerBound}}{Theorem 3}}

Here, we show that every connected graph $G$ with $n \geq 3$ vertices and $m \geq 0$ edges satisfies 
\[
\Ucr(G) \geq \left\lceil\frac{m}{f(n,m)}\right\rceil
\]
where $f(n,m) = \left(3n-5+\sqrt{(3n-5)^2-4m}\right)/2$.

Let $\mathcal{D}(G) = \{D_1,\dots,D_k\}$ be an uncrossed collection of drawings of $G$.
For every $i \in [k]$, let $D'_i$ be a subdrawing of $D_i$ containing only edges of $D_i$ that are uncrossed in $D_i$.
By \Cref{lem-basicProperties}, each drawing $D'_i$ is then a plane graph with the property that every edge of $G$ that is not an edge of $D'_i$ is contained in a single face of $D'_i$.
Moreover, since $G$ is connected, we can assume without loss of generality by this lemma that each $D'_i$ represents a connected subgraph of~$G$ as to bound $\Ucr(G)$ from below it suffices to consider drawings $D'_i$ with as many edges as possible.

Fix an arbitrary $i \in [k]$.
The number of vertices of $D'_i$ equals $n$.
We use $m_i$ to denote the number of edges of $D'_i$ and we will show that $m_i \leq f(n,m)$.

We set $\mathcal{F}_i$ to be the set of faces of $D'_i$ and $f_i=|\mathcal{F}_i|$.
For a face $F$ of $D'_i$, we use $v(F)$ for the number of vertices of $D'_i$ that are contained in the boundary of $F$ and we write $|F|$ for the number of times we meet an interior of an edge of $D'_i$ when traversing $F$ along its boundary.
That is, $|F|$ is the length of the facial walk.
Note that each edge can be counted once or twice in $|F|$ and so we have $v(F) \leq |F|$ as $D'_i$ represents a connected subgraph of $G$.
We assume that at least one edge of $F$ is counted once in $|F|$ and that $v(F) \geq 3$ for every face $F$ as otherwise there is only a single face in $\mathcal{F}_i$ and $D'_i$ is a tree with $m_i \leq n-1 \leq f(n,m)$ for $n \geq 3$. 
Also, observe that 
\begin{equation}
\label{eq-faces1}
\sum_{F \in \mathcal{F}_i} |F| = 2m_i
\end{equation}
as every edge is incident to a face of $D'_i$ from the left and from the right.

Since every edge of $G$ that is not an edge of $D'_i$ is contained in a single face of $D'_i$, we have
\begin{equation}
\label{eq-faces2}
\sum_{F \in \mathcal{F}_i} \left(\binom{v(F)}{2} - v(F)\right) \geq m-m_i.
\end{equation}
This is because vertices of each face $F$ can span up to $\binom{v(F)}{2}$ edges of $D_i$ and at least $v(F)$ pairs of vertices of $D_i$ are already used for edges of $D'_i$ as each face $F$ contains an edge that is counted only once in~$|F|$.
The left hand side of~\eqref{eq-faces2} can be rewritten as 
\[\frac{1}{2}\sum_{F \in \mathcal{F}_i} v(F)(v(F)-3).\]
Since $v(F) \geq 3$ and $|F| \geq v(F)$ for every face $F$ from $\mathcal{F}_i$, we obtain 
\[
\frac{1}{2}\sum_{F \in \mathcal{F}_i} |F|(|F|-3)\geq m-m_i.
\] 
Since $|F|-3 \geq 0$, the left-hand side can be bounded from above by
\[\frac{1}{2}\left(\sum_{F \in \mathcal{F}_i} |F|\right)\left(\sum_{F \in \mathcal{F}_i} (|F|-3)\right) = m_i(2m_i-3f_i)\]
where we used~\eqref{eq-faces1} twice.
Altogether, we obtain $m_i(2m_i-3f_i) \geq m-m_i$, which can be rewritten as
\[
f_i \leq \frac{2m_i}{3} - \frac{m-m_i}{3m_i}.
\]
Plugging this estimate into Euler's formula $n-m_i+f_i=2$, we get
\[m_i \leq 3n - 5 -\frac{m}{m_i},\]
which after solving the corresponding quadratic inequality for $m_i$ gives the final estimate
\[m_i \leq (3n-5+\sqrt{(3n-5)^2-4m})/2 = f(n,m).\]

Since $i$ was arbitrary, we see that each drawing $D'_i$ contains at most $f(n,m)$ edges of~$G$ and therefore, we indeed have 
\[k \geq  \Ucr(G) \geq \left\lceil\frac{m}{f(n,m)}\right\rceil.\]{\flushright\vspace{-2.4em}\qed} %

\section{Proof of \texorpdfstring{\Cref{theorem:hard:ecr}}{Theorem 4}}
\label{sec:ecrNPHproof}

\begin{figure}
	\centering
	\begin{tikzpicture}
		\boundingbox
		\foreach \n [evaluate=\n as \a using 90-360*(\n-1)/8] in {1,2,...,8}{\node[origV] (\n) at (\a:\radius) {$\n$};}%
		\foreach \n in {1,2,...,5,8}{\gateedge[red!70!black]{(\n)};}%
		\foreach \n in {6,7}{\gateedge[black]{(\n)};}%
		\node[centerV] (m) at (0,0) {\bf c};

		\origedgeDetailed{1}{2}{}{0.5}
		\origedge{1}{8}{}{0.5}
		\origedge{2}{3}{}{0.5}
		\origedge{3}{4}{}{0.5}
		\origedge{3}{5}{bendLmedbig}{0.5}
		\origedge{3}{8}{bendRbig}{0.5}
		\origedge{5}{6}{}{0.5}
		\origedge{5}{7}{bendLmedbig}{0.5}
		\origedge{5}{8}{bendLbig}{0.5}
		\origedge{6}{7}{}{0.5}
		\origedge{7}{8}{}{0.5}

		\origedge[red!70!black]{1}{4}{bendLsmall}{0.5}
		\origedge[red!70!black]{1}{7}{}{0.5}
		\origedge[red!70!black]{2}{6}{bendLsmall}{0.5}
		\origedge[red!70!black]{3}{7}{bendRsmall}{0.5}
		\origedgeDetailed[red!70!black]{4}{6}{}{0.6}
	\end{tikzpicture}
	\caption{
		\ecrProblem~instance for proof of \texorpdfstring{\Cref{theorem:hard:ecr}}{Theorem 4} in \cref{sec:ecrNPHproof}.
		Red edges are crossed.
		Thick edges represent $M$-bundles corresponding to edges from the {\sc Maximum Outerplanar Subgraph} instance $G$, see in detail the edges between $1$ and $2$ and between $6$ and $4$.
		The dashed edges and $c$ form the central star.
	}
	\label{fig:hard:ecr-red}
\end{figure}
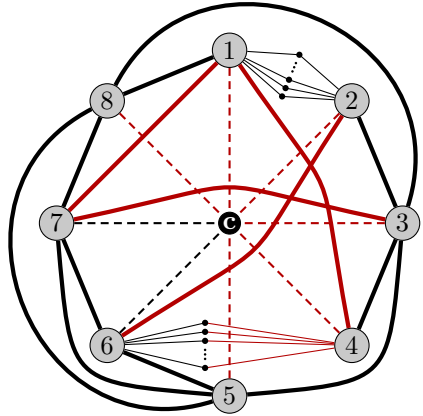

In this section, we prove that \ecrProblem{} is \NP-complete.
Membership of this problem in the class \NP{} is trivial.
To show \NP-hardness, we reduce from the following \NP-complete problem ~\cite{MaxOuterplanarsubgraphNPhard1,MaxOuterplanarsubgraphNPhard2}.

\Prob{Maximum Outerplanar Subgraph}{A graph $G = (V,E)$ and a positive integer $k$.}{Is there an outerplanar subgraph of $G$ with at least $k$ edges?}

Assume an instance of {\sc Maximum Outerplanar Subgraph}.
Let $M > |V|$, say $M=2|V|$, and $k'=|E|-k$.
We augment $G$ into a graph $G'$, and show that $G'$ can be drawn with at most $M k'+|V|$ crossed edges, if and only if $G$ admits an outerplanar subgraph with at least $k$ edges.
The graph $G'$ is obtained via two augmenting steps:
We add a \emph{central star}, i.e., a vertex with an edge to each original vertex of $G$.
Then, we replace each original edge in $G$ by $M$ parallel paths of length two, which we call an \emph{$M$-bundle}.
An example of this transformation can be seen in \cref{fig:hard:ecr-red}.

Suppose there is a drawing of $G'$ with at most $M k' + |V|$ crossed edges.
We want to modify this drawing into a drawing of $G$.
To this end, we first remove every path belonging to an $M$-bundle, if either of its two edges is crossed.
We also remove the central vertex and all of its incident edges.
All remaining edges are uncrossed and belong to an $M$-bundle path.
As there are at most $M k' + |V| < M (k'+1)$ crossed edges in the drawing, for at least $|E|-k'=k$ edges from $G$ there is at least one path of its corresponding $M$-bundle that is not removed.
We contract for each edge of $G$ one edge of one of the remaining paths of its $M$-bundle and remove all other $M$-bundle paths.
The vertices from $G$ all share the face created by removing the central vertex and all vertices from $M$-bundles are either contracted or removed.
Thus, we have an outerplanar drawing of a subgraph of $G$ with at least $k$ edges.

Similarly, for every outerplanar subgraph $H$ of $G$ with at least $k$ edges we can construct a drawing of $G'$ with at most $M k' + |V|$ crossed edges.
First, we draw $H$ in an outerplanar embedding, then we draw the central star into the outer face.
Next, we draw the at most $|E|-k = k'$ remaining edges of $G$ in such a way that they only cross one another and the edges of the central star.
Finally, we replace every edge of $G$ with an $M$-bundle.
The newly added vertices are positioned in such a way that at most one of the edges of each path is crossed.
Therefore, there are at most $M k'$ crossed edges from the $M$-bundles and at most $|V|$ crossed edges from the central star, for a total of at most $M k' + |V|$ crossed edges.
\hfill\qed %

\section{Proof of \texorpdfstring{\Cref{thm-hardness}}{Theorem 6}}

\begin{figure}
	\centering
	\begin{subfigure}[t]{.45\textwidth}
		\centering
		\begin{tikzpicture}
			\boundingbox
			\foreach \n/\a in {1/90,2/45,3/360,4/315,5/270,6/225,7/180,8/135}{
				\node[origV] (\n) at (\a:\radius) {$\n$};
				\node[helperV] (m\n) at ($(\n)!.5!(0,0)$) {};
			}
			\node[centerV] (m) at (0,0) {\bf c};

			\foreach \u/\v in {m1/1, m2/2, m3/3, m4/4, m5/5, m8/8}
				\draw[splitE,        red!70!black] (\u) to (\v);
			\foreach \u/\v in {m1/m, m3/m, m4/m, m5/m, m7/m, m8/m}
				\draw[splitE,        black] (\u) to (\v);
			\foreach \u/\v in {m2/m, m6/6, m6/m, m7/7}
				\draw[splitE,        gray] (\u) to (\v);

			\draw[origE,         gray] (1) to (2);
			\draw[origE,         gray] (1) to (8);
			\draw[origE] (2) to (3);
			\draw[origE,         gray] (3) to (4);
			\draw[origE,bendLmedbig,         gray] (3) to (5);
			\draw[origE,bendRbig] (3) to (8);
			\draw[origE] (5) to (6);
			\draw[origE,bendLmedbig] (5) to (7);
			\draw[origE,bendLbig] (5) to (8);
			\draw[origE] (6) to (7);
			\draw[origE] (7) to (8);

			\draw[origE,red!70!black,bendLsmall] (1) to (4);
			\draw[origE,red!70!black] (1) to (7);
			\draw[origE,red!70!black,bendLmedbig] (2) to (6);
			\draw[origE,red!70!black,bendRmedbig] (3) to (7);
			\draw[origE,red!70!black] (4) to (6);
		\end{tikzpicture}
		\caption{\centering First drawing.}
		\label{fig:hard:uc-ot:d1}
	\end{subfigure}
	\hfill
	\begin{subfigure}[t]{.45\textwidth}
		\centering
		\begin{tikzpicture}
			\boundingbox
			\foreach \n/\a in {3/225,6/360}{\node[origV,blue,scale=1.15] at (\a:\radius) {$\n$};}%
			\foreach \n/\a in {1/90,2/45,6/360,4/315,5/270,3/225,7/180,8/135}{
				\node[origV] (\n) at (\a:\radius) {$\n$};
				\node[helperV] (m\n) at ($(\n)!.5!(0,0)$) {};
			}
			\node[centerV] (m) at (0,0) {\bf c};

			\foreach \u/\v in {m1/1, m2/2, m3/3, m4/4, m5/5, m8/8}
				\draw[splitE,        black] (\u) to (\v);
			\foreach \u/\v in {m1/m, m3/m, m4/m, m5/m, m7/m, m8/m}
				\draw[splitE, red!70!black] (\u) to (\v);
			\foreach \u/\v in {m2/m, m6/6, m6/m, m7/7}
				\draw[splitE,         gray] (\u) to (\v);

			\draw[origE,         gray] (1) to (2);
			\draw[origE,         gray] (1) to (8);
			\draw[origE,red!70!black,bendRsmall] (2) to (3);
			\draw[origE,         gray,bendRmedbig] (3) to (4);
			\draw[origE,         gray] (3) to (5);
			\draw[origE,red!70!black,bendRmed] (3) to (8);
			\draw[origE,red!70!black,bendLmed] (5) to (6);
			\draw[origE,red!70!black,bendRmed] (5) to (7);
			\draw[origE,red!70!black,bendRsmall] (5) to (8);
			\draw[origE,red!70!black,bendLsmall] (6) to (7);
			\draw[origE,         gray] (7) to (8);

			\draw[origE,bendLbig] (1) to (4);
			\draw[origE,bendRmedbig] (1) to (7);
			\draw[origE] (2) to (6);
			\draw[origE] (3) to (7);
			\draw[origE] (4) to (6);
		\end{tikzpicture}
		\caption{\centering Second drawing. Vertices~$3$~and~$6$~swapped~places.}
		\label{fig:hard:uc-ot:d2}
	\end{subfigure}
	\caption{
		An instance of the reduction from {\sc OuterThickness} to \UcrProblem.
        The original graph $G$ in this instance is drawn with solid edges and has outerthickness~$2$ (as the two subdrawings in solid black and gray edges certify).
		The dashed edges and black vertices form the central star around $c$ added to $G$ in the reduction.
        In each drawing, all crossed edges are red and uncrossed edges of the particular drawing are black, and gray edges are uncrossed in both drawings.
	}
	\label{fig:hard:uc-ot}
\end{figure}
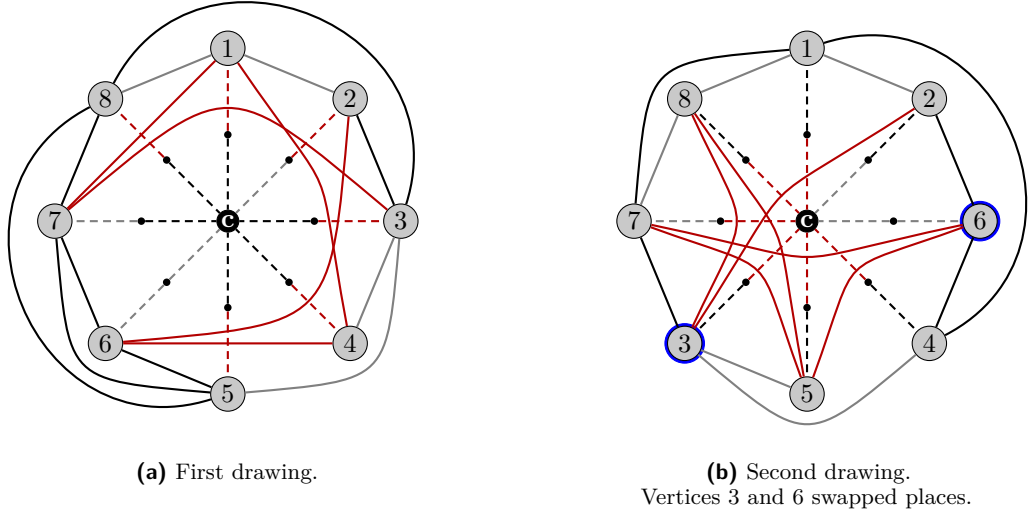

We show that if {\sc Outerthickness} is \NP-hard, then \UcrProblem{} is \NP-hard as well using a reduction from {\sc Outerthickness} to \UcrProblem.

The reduction employs similar arguments as used in \Cref{sec:ecrNPHproof}.
Let $(G,k)$ be an instance of the problem {\sc Outerthickness}.
We augment the input graph $G$ into a graph $G'$ by adding a vertex and connecting each vertex of $G$ to it with a path of length two.
We call the added structure the \emph{central star}.
See \cref{fig:hard:uc-ot} for an example of this transformation.

Let $D$ be a drawing of $G'$.
Consider the uncrossed subdrawing $D'_G$ consisting of the vertices and all uncrossed edges from $G$.
As there is a path in $D\setminus D'_G$ between each two vertices from $G$, we know that $D'_G$ is outerplanar.
Thus, if $\Ucr(G') \leq k$ and $D_1,\dots,D_k$ is an uncrossed collection of $G'$, then the respective subdrawings restricted to $G$ decompose $G$ into $k$ outerplanar graphs.

Conversely, if $G$ can be decomposed into $k\geq2$ outerplanar subgraphs $G_1, \dots, G_k$, then we can construct an uncrossed collection $D_1,\dots,D_k$ of $G'$ in the following way:
In every drawing $D_i$, we first draw $G_i$ as an outerplanar graph and we embed the central star in the outer face.
Then, we draw the remaining original edges in such a way that they only cross each other and edges from the central star.
In $D_1$, all crossings on the central star lie on edges incident to vertices of $G$, and in all other drawings, the crossings on the central star involve only edges incident to the universal vertex.
This way we assure that also every edge of the central star is uncrossed in some drawing.
\hfill \qed%

\section{Conclusions and Open Problems}

We provided exact values of the uncrossed number for complete and complete bipartite graphs.
The hypercube graphs form another natural graph class to consider as their outerthickness and thickness were determined exactly; see~\cite{guyNow90I,klinert67}.
However, we are not aware of any formula for the uncrossed number for the hypercube graphs.

\begin{question}
Determine the exact value of the uncrossed number for the hypercube graphs.
\end{question}

In \Cref{thm-lowerBound}, we determined a general lower bound on $\Ucr(G)$ in terms of the number of the edges and vertices of $G$ by showing $\Ucr(G)\ge\lceil\frac{m}{c n}\rceil - O(n) - O(m)$ for some constant $c$ with $0 < c \leq 3$.
In particular, we argued that the smallest obtainable constant $c$ is approximately $2.82$ for the case of dense $n$-vertex graphs with $\varepsilon n^2$ edges where $\varepsilon \in (0,1/2)$ is a fixed constant.
Can one obtain a better leading constant in the general lower bound on $\Ucr(G)$ for such dense graphs $G$?

We also propose investigating other properties of the uncrossed number. 
We conjecture that the uncrossed number can be arbitrarily far apart from the outerthickness despite them being quite similar on the graph classes we mainly investigated in this paper.
In fact, it follows from our results that the difference between the outerthickness and the uncrossed number of complete and complete bipartite graphs is never larger than one.

\begin{conjecture}
  For every positive integer $k$, there is a graph $G$ such that 
  \[\theta_o(G)-\Ucr(G) \geq k.\]
\end{conjecture}

Lastly, it would be interesting to finally settle the computational complexity of the outerthickness problem.
We conjecture that the {\sc Outerthickness} problem is \NP-hard.
Note that if true, this would also settle the computational complexity of \UcrProblem{} by \Cref{thm-hardness}.

\paragraph{Acknowledgements.}
We would like to thank the organizers of the Crossing Numbers Workshop 2023 where this research was initiated.
We also would like to thank the anonymous reviewer who pointed out a simplification in the last case of our proof of Theorem~\ref{thm:completeBip}.

\bibliography{cro.bib}

\end{document}